\documentclass[final,5p,times,twocolumn]{elsarticle}

\usepackage{amsfonts,amssymb}
\usepackage{mathtools}
\usepackage{enumitem}
\usepackage{graphicx}
\usepackage[hidelinks]{hyperref}

\newcommand{\C}{\mathbb C}
\newcommand{\R}{\mathbb R}
\newcommand{\Q}{\mathbb Q}
\newcommand{\Rea}{\operatorname{Re}}
\newcommand{\abs}[1]{\left|#1\right|}
\newcommand{\norm}[1]{\left\|#1\right\|}

\newtheorem{theorem}{Theorem}[section]
\newtheorem{lemma}[theorem]{Lemma}
\newtheorem{proposition}[theorem]{Proposition}
\newtheorem{corollary}[theorem]{Corollary}
\newenvironment{proof}[1][Proof]
  {\par\smallskip\noindent\textit{#1.}\ }
  {\hfill$\square$\par\smallskip}

\begin{document}
\begin{frontmatter}
\title{Static output-feedback stabilization is NP-hard}
\author{Johan L\"ofberg}
\address{Division of Automatic Control\\Link\"oping University, Sweden
(e-mail: \texttt{johan.lofberg@liu.se})}
\date{August 17, 2026}

\begin{abstract}
We prove that unrestricted static output-feedback stabilization is NP-hard, already in the single-input multiple-output case. More precisely, we give a polynomial-time many-one reduction from the NP-complete Betweenness problem. Given an instance with $n$ elements and $m$ betweenness constraints, the reduction constructs integer matrices
$A\in\mathbb Z^{D\times D}$, $B\in\mathbb Z^{D\times1}$, and $C\in\mathbb Z^{n\times D}$ such that the instance is satisfiable if and only if there exists an unrestricted real row vector $K\in\R^{1\times n}$ for which $A+BKC$ is Hurwitz. The state dimension is $D=6(m+n)$, and every matrix entry has logarithmic bit length in the input size. The reduction uses one universal cubic source polynomial whose Routh determinant enforces the sign condition encoding betweenness, while fixed anchor values provide a mechanism for excluding spurious stabilizing gains of arbitrarily large magnitude. Frequency separation then compresses the simultaneous source conditions into one common affine polynomial, which is realized directly as the characteristic polynomial of $A+BKC$. As a by-product, deciding whether an affinely parameterized rational polynomial family contains a Hurwitz member on a rational box is NP-hard.
\end{abstract}
\end{frontmatter}

\section{Problem statement and main result}\label{sec:problem}

A square matrix is \emph{Hurwitz} if all its eigenvalues have strictly negative real part, and a polynomial is Hurwitz if all its roots have strictly negative real part. For a plant with rational matrices $A\in\Q^{D\times D}$, $B\in\Q^{D\times1}$, and $C\in\Q^{r\times D}$, and with the output feedback gain ranging over $K\in\R^{1\times r}$, we study the single-input, multiple-output (SIMO) static output-feedback decision predicate
\[
\mathsf{SOF}_{1}(A,B,C):\qquad
\exists K\colon A+BKC\ \text{is Hurwitz}.
\]

Static output-feedback stabilization is a fundamental and longstanding problem in control \cite{SYRMOS1997125}. Blondel and Tsitsiklis proved NP-hardness when the gain is subject to prescribed interval constraints, but stressed that this says little about the unrestricted problem: interval-constrained state-feedback stabilization is also NP-hard, whereas unconstrained state-feedback stabilization is solvable in polynomial time. They explicitly left the complexity of unrestricted static output-feedback stabilization open, and their subsequent survey again highlighted its unresolved status \cite{BlondelTsitsiklis1997,BlondelTsitsiklis2000}.

Other hardness results concern closely related but different problems. Toker and Özbay proved NP-hardness of general bilinear matrix inequality
feasibility and, independently, of simultaneous stabilization of
multiple plants by a common static output-feedback gain
\cite{TokerOzbay1995}. They explicitly noted that the latter does not
settle the complexity of static output-feedback stabilization for a
single plant. Fu and Luo subsequently proved NP-hardness of a
nonconvex matrix-inequality feasibility problem arising in fixed-order
output-feedback design \cite{FuLuo1997}. Later work proved that static
output-feedback pole placement is NP-hard, while noting that this does
not imply NP-hardness of stabilization \cite{Fu2004}.

These distinctions are essential. NP-hardness of related problems, or nonconvexity and poor computational complexity of certain solution methods, does not by itself establish NP-hardness of unrestricted static output-feedback stabilization. These related results sometimes lead to a conflation of the different questions, and one often encounters statements that ``static output-feedback stabilization is NP-hard'' without specifying the precise problem formulation.

To the best of our knowledge, the theorem below resolves this unrestricted hardness question.

\begin{theorem}[Main theorem]\label{thm:main}
The decision problem $\mathsf{SOF}_{1}$ is NP-hard. More precisely, from a Betweenness instance with $n$ elements and $m$ constraints one can construct in polynomial time integer matrices
\[
A\in\mathbb Z^{D\times D},\qquad
B\in\mathbb Z^{D\times1},\qquad
C\in\mathbb Z^{n\times D},
\]
with
\[
D=6(m+n),
\]
such that
\[
\begin{aligned}
&\text{the Betweenness instance is satisfiable}\\
&\qquad\Longleftrightarrow\quad
\exists K\in\R^{1\times n}:\ A+BKC\text{ is Hurwitz}.
\end{aligned}
\]
Every matrix entry has $O(\log L)$ bits, where $L$ is the encoding length of the Betweenness instance, and the complete dense matrix encoding has $O(L^2\log L)$ bits.
\end{theorem}

Hurwitz stability is invariant under transposition, and
\[
(A+BKC)^\top=A^\top+C^\top K^\top B^\top.
\]
Thus the theorem also gives NP-hardness for multiple-input, single-output static output feedback. Padding $B$ with zero columns and $C$ with zero rows embeds the SIMO instance into a general MIMO instance without changing feasibility, so unrestricted MIMO stabilization is NP-hard as well.

\subsection{The Betweenness problem}

A \emph{Betweenness} instance consists of a finite set $V$ and a collection $T$ of ordered triples $(a,b,c)$ of distinct elements of $V$. A total order satisfies $(a,b,c)$ if $b$ lies strictly between $a$ and $c$, i.e., either
\[
a<b<c
\qquad\text{or}\qquad
c<b<a.
\]
The decision problem asks whether there exists a total order of $V$ that satisfies every triple in $T$. Opatrny proved this problem NP-complete \cite{Opatrny1979}. We write
\[
n=|V|,\qquad m=|T|,
\]
and let $L$ be the binary encoding length of the instance. Then $n,m=O(L)$. Trivial empty instances can be handled separately, so we assume $n,m\ge1$.

We use a convenient anchored form. Add two additional elements called anchor values $\mathsf L,\mathsf R$ and, for every $i\in V$, append the ordinary Betweenness constraint
\[
(\mathsf L,i,\mathsf R).
\]
This restriction remains NP-complete: any satisfying order of the original instance extends by placing $\mathsf L$ first and $\mathsf R$ last, while deleting the anchors from a satisfying augmented order leaves every original constraint satisfied. Thus the transformation is a polynomial-time equivalence.

Throughout, $n=|V|$ and $m=|T|$ continue to denote the counts in the original instance. The anchored instance has $n+2$ elements and $m+n$ constraints, but the anchors have fixed numerical values and introduce neither gain variables nor additional free parameters.

\subsection{Proof blueprint}

The reduction assigns one unrestricted real parameter $x_i$ to each element $i\in V$. Put
\[
N:=2(n+1).
\]
The added anchor constraints are represented by fixing the numerical values of $\mathsf L$ and $\mathsf R$ at $0$ and $N$, respectively. These fixed anchors do not restrict $x$ (the future feedback gain); rather, their source-cubic stability conditions will force every stabilizing realization into a bounded box. Section~\ref{sec:source} uses one universal cubic
\[
q(s;u,v)=s^3+(N+u)s^2+N(N+v)s+N^2(N+u+v)+1.
\]
Its only nontrivial cubic Routh determinant is
\[
N(N+u)(N+v)-\bigl(N^2(N+u+v)+1\bigr)=Nuv-1.
\]
For an original triple $(a,b,c)$ we substitute
\[
u=x_b-x_a,\qquad v=x_c-x_b,
\]
so stability forces $x_b$ to lie between $x_a$ and $x_c$. For the anchor $(\mathsf L,i,\mathsf R)$ we obtain
\[
u=x_i,\qquad v=N-x_i,
\]
so the same Routh determinant becomes $Nx_i(N-x_i)-1$; Hurwitz stability of this anchor cubic therefore forces $0<x_i<N$. Equivalence between the original Betweenness instance and the simultaneous Hurwitz stability of all $m+n$ source cubics is exact, and Lemma~\ref{lem:global-defect} gives a uniform defect for no-instances, exploiting the compactness of the closed interval $[0,N]$.

The main step in Section~\ref{sec:aggregation} compresses the $m+n$ simultaneous cubic stability conditions into a single real monic polynomial $P(s,x)$ whose coefficients are integer and affine in $x$. Conjugate copies of the source cubics are shifted to widely separated imaginary frequencies and the exact product is truncated to the terms that are constant or linear in $x$. While doing this, the initially exploited compact region for $x$ is lost, but a temporary  bounded region can be chosen large enough to accommodate the necessary analysis. On a sufficiently large auxiliary bounded region, Rouch\'e's theorem transfers the relevant local root counts from each source cubic to $P$.

Section~\ref{sec:escape} then crucially removes the auxiliary-box restriction.
The anchor perturbation $h(s)=s(s-N)$ has a fixed open-right-half-plane
zero at $s=N$, and a local Rouch\'e argument shows that every sufficiently
large gain gives the aggregate an open-right-half-plane root as needed.

Finally, Section~\ref{sec:realization} realizes the affine aggregate directly as
\[
\det\bigl(sI-(A+BKC)\bigr),
\qquad K=[x_1\ \cdots\ x_n].
\]
The resulting realization is sparse and is constructed directly from the Betweenness incidence data and a few polynomially bounded constants.

The reduction is largely self-contained and uses only the cubic Routh--Hurwitz criterion \cite{Gantmacher1959}, Rouch\'e's theorem \cite{Conway1978}, and elementary inequalities and algebra. The only non-standard external result is the NP-completeness of Betweenness \cite{Opatrny1979}.

\subsection{Root analysis tools}

\begin{lemma}[Scaled cubic root bound]\label{lem:scaled-cubic-bound}
Let $p(s)=s^3+as^2+bs+c$. Every root $\zeta$ of $p$ satisfies
\[
\abs\zeta\le2\max\{1,\abs a,\abs b^{1/2},\abs c^{1/3}\}.
\]
\end{lemma}

\begin{proof}
Since $p(\zeta)=0$, the triangle inequality gives
\[
\abs\zeta^3\le\abs a\abs\zeta^2+\abs b\abs\zeta+\abs c.
\]
Set $r:=\abs\zeta$. If
\[
r>2\max\{1,\abs a,\abs b^{1/2},\abs c^{1/3}\},
\]
then
\[
1\le\frac{\abs a}{r}+\frac{\abs b}{r^2}+\frac{\abs c}{r^3}
<\frac12+\frac14+\frac18<1,
\]
which is a contradiction.
\end{proof}

\begin{lemma}[Local root-count transfer]\label{lem:root-transfer}
Let $\mathcal{S}\subset\C$ be a bounded domain with piecewise smooth boundary, and let $P,Q,G$ be polynomials. Suppose that
\[
G(s)\ne0\qquad(s\in\mathcal{S})
\]
and
\[
\abs{P(s)-Q(s)G(s)}<\abs{Q(s)G(s)}
\qquad(s\in\partial\mathcal{S}).
\]
Then $P$ and $Q$ have the same number of roots in $\mathcal{S}$, counted with multiplicity, and neither has a root on $\partial\mathcal{S}$.
\end{lemma}

\begin{proof}
Rouch\'e's theorem gives the same root count for $P$ and $QG$. Since $G$ is nonvanishing on $\mathcal{S}$, the roots of $QG$ in $\mathcal{S}$ are exactly those of $Q$, with the same multiplicities. The strict boundary inequality forces $QG$ to be nonzero there, and the reverse triangle inequality then gives $\abs P>0$ on $\partial\mathcal{S}$.
\end{proof}

\section{Anchored Betweenness source cubics}\label{sec:source}

Let $V=\{1,\ldots,n\}$ and let $T$ contain the $m$ original Betweenness triples. Set
\begin{equation}\label{eq:N}
N=2(n+1)\qquad(N\ge4).
\end{equation}
The reduction assigns one real coordinate $x_i$ to each element $i\in V$. Write
\[
x=(x_1,\ldots,x_n)\in\R^n.
\]
The anchor values are fixed at $0$ and $N$ and introduce no gain variables.

Define the universal source cubic which will do all the heavy lifting in representing Betweenness constraints. It is a monic cubic in $s$ with two real parameters $u,v$:
\begin{equation}\label{eq:universal-cubic}
q(s;u,v)
=s^3+(N+u)s^2+N(N+v)s+\bigl(N^2(N+u+v)+1\bigr).
\end{equation}
For a monic cubic $s^3+as^2+bs+c$, the Routh--Hurwitz criterion is
\[
a>0,\qquad b>0,\qquad c>0,\qquad \Delta = ab-c>0.
\]
For \eqref{eq:universal-cubic}, the nontrivial determinant $\Delta$ cancels conveniently:
\begin{equation}\label{eq:delta-uv}
\Delta=N(N+u)(N+v)-\bigl(N^2(N+u+v)+1\bigr)=Nuv-1.
\end{equation}
For a triple $t=(a,b,c)\in T$, put
\begin{equation}\label{eq:uv-triple}
u_t=x_b-x_a,\qquad v_t=x_c-x_b
\end{equation}
and define $q_t^B(s,x)=q(s;u_t,v_t)$. By \eqref{eq:delta-uv}, Hurwitz stability of $q_t^B$ requires
\[
N(x_c-x_b)(x_b-x_a)>1,
\]
which encodes the Betweenness constraint that $x_b$ lies strictly between $x_a$ and $x_c$. Its \emph{baseline} and \emph{affine perturbation} are
\begin{align}
\bar q_B(s)&:=q_t^B(s,0)=s^3+Ns^2+N^2s+N^3+1,\label{eq:qB0}\\
q_t^B(s,x)-\bar q_B(s)
&=-x_a(s^2+N^2)+x_b(s^2-Ns)+x_c(Ns+N^2).\label{eq:triple-pert}
\end{align}

For each $i\in V$, the appended anchor $(\mathsf L,i,\mathsf R)$ is represented by $u=x_i$ and $v=N-x_i$. Its cubic is
\begin{align}
q_i^A(s,x_i)
&=s^3+(N+x_i)s^2+N(2N-x_i)s+(2N^3+1)\notag\\
&=\bar q_A(s)+x_i h(s),\label{eq:qA}
\end{align}
where
\begin{equation}\label{eq:qA0-h}
\bar q_A(s)=s^3+Ns^2+2N^2s+2N^3+1,~ h(s)=s^2-Ns=s(s-N).
\end{equation}
Its Routh determinant is
\begin{equation}\label{eq:anchor-delta}
\Delta_i^A=Nx_i(N-x_i)-1.
\end{equation}
thus Hurwitz stability of the anchor cubic requires $0<x_i<N$. 

We can now prove our first step in our chain of equivalences by showing that our polynomials indeed encode Betweenness problems via stability properties.

\begin{proposition}[Exact source equivalence]\label{prop:source}
The supplied Betweenness instance is satisfiable if and only if there exists $x\in\R^n$ for which all $m+n$ source cubics $q_t^B$ and $q_i^A$ are Hurwitz.
\end{proposition}

\begin{proof}
Suppose first that a satisfying total order is given. Let $\pi(k)$ denote the $k$th element in that order, and set
\[
x_{\pi(k)}=2k,\qquad k=1,\ldots,n.
\]
Then $2\le x_i\le2n<N$. For every anchor,
\[
Nx_i(N-x_i)\ge4Nn,
\]
so \eqref{eq:anchor-delta} is at least $15$, since $N\ge4$ and $n\ge1$, and all anchor coefficients are positive. For a satisfied original triple, $u_t,v_t$ are nonzero even integers with the same sign, hence $u_tv_t\ge4$ and \eqref{eq:delta-uv} is at least $15$, since $N\ge4$. Moreover
\[
N+u_t\ge4,\qquad N+v_t\ge4,
\]
and since $u_t+v_t=x_c-x_a\ge-2(n-1)$,
\[
N^2(N+u_t+v_t)+1\ge4N^2+1>0.
\]
Thus every source cubic is Hurwitz.

Conversely, if every anchor cubic is Hurwitz, \eqref{eq:anchor-delta} gives
\[
Nx_i(N-x_i)>1,
\]
hence $0<x_i<N$ for every $i$. For an original triple all three nonleading coefficients are then positive, and Hurwitz stability gives
\[
N(x_b-x_a)(x_c-x_b)>1>0.
\]
Thus $x_b$ lies strictly between $x_a$ and $x_c$. Sorting the real numbers $x_i$ yields a total order satisfying every original triple. If some unrelated elements have equal values, break those ties arbitrarily; no constrained triple contains a tie because its product is strictly positive.
\end{proof}

The same structure gives a uniform defect for no-instances, globally over all real parameter vectors.

\begin{lemma}[Global no-instance defect]\label{lem:global-defect}
If the Betweenness instance is unsatisfiable, then for every $x\in\R^n$ there is a source cubic whose Routh determinant satisfies
\[
\Delta\le-1.
\]
The selected cubic can always be chosen with constant coefficient $c\ge1$.
\end{lemma}

\begin{proof}
If some $x_i\notin[0,N]$, select its anchor. Then $x_i(N-x_i)\le0$, so \eqref{eq:anchor-delta} gives $\Delta_i^A\le-1$, while its constant coefficient is $2N^3+1$.

Otherwise all $x_i\in[0,N]$. If every original triple had
\[
(x_b-x_a)(x_c-x_b)>0,
\]
sorting the $x_i$ would satisfy every Betweenness constraint, contradicting unsatisfiability. Hence some original triple has product at most zero and therefore $\Delta\le-1$. For that triple
\[
c=N^2(N+x_c-x_a)+1\ge1.
\]
\end{proof}

\subsection{Uniform source estimates on an auxiliary bounded region}

The forthcoming aggregation needs uniform root and conditioning bounds on a bounded set. Define
\begin{equation}\label{eq:M0}
M_0:=100N,
\qquad
\mathcal B:=\{x\in\R^n:\norm{x}_\infty\le M_0\}.
\end{equation}
Note that this box is only an analytical device to enable us to analyze the problem in $\mathcal B$ and its complement in two stages; the final gain remains unrestricted. Introduce the explicit constants
\begin{align}
A_0&:=N+2M_0=201N,\label{eq:A0}\\
R&:=2A_0=402N,\qquad R_0:=R+1,\label{eq:R-R0}\\
\mu&:=\frac{1}{2A_0^2+1}.\label{eq:mu}
\end{align}

\begin{lemma}[Root disks]\label{lem:root-disks}
For $x\in\mathcal B$, every root of every source cubic and of either baseline cubic $\bar q_B,\bar q_A$ lies in $\abs s\le R$.
\end{lemma}

\begin{proof}
For an original triple, $|u|,|v|\le2M_0$, hence the coefficients of $q_t^B$ satisfy
\[
\begin{aligned}
|N+u|&\le A_0,\qquad |N(N+v)|\le201N^2,\\
|N^2(N+u+v)+1|&\le401N^3+1.
\end{aligned}
\]
For an anchor the corresponding bounds are $101N$, $102N^2$, and $2N^3+1$. Thus, for every source cubic and baseline cubic, $A_0$ dominates $1$, the magnitude of the quadratic coefficient, the square root of the linear coefficient, and the cube root of the constant coefficient. Lemma~\ref{lem:scaled-cubic-bound} gives the common bound $R=2A_0$.
\end{proof}

\begin{lemma}[Uniform imaginary-axis separation]\label{lem:imag-gap}
Let $p(s)=s^3+as^2+bs+c$ be either a source cubic at the satisfying realization from Proposition~\ref{prop:source}, or the bad source cubic selected by Lemma~\ref{lem:global-defect} for some $x\in\mathcal B$. Then
\[
\abs{p(i\omega)}\ge\mu
\qquad(\omega\in\R).
\]
At the satisfying realization, $p$ is Hurwitz. In the selected no-instance case, $p$ has an open-right-half-plane root.
\end{lemma}

\begin{proof}
In every relevant case $c\ge1$, $|a|\le A_0$, and
\[
|ab-c|\ge1.
\]
At a satisfying realization the latter quantity is at least $15$. Write
\[
p(i\omega)=R(\omega)+iI(\omega),\qquad
R=c-a\omega^2,\qquad I=\omega(b-\omega^2),
\]
and put $\delta=ab-c$. For $\omega\ne0$ the exact identity
\begin{equation}\label{eq:cubic-imag-identity}
\delta=a\frac{I}{\omega}-R
\end{equation}
holds. If $|\omega|\le(2A_0)^{-1/2}$, then
\[
R\ge c-|a|\omega^2\ge\frac12,
\]
so $|p(i\omega)|\ge1/2>\mu$. Otherwise,
\[
\begin{aligned}
1\le|\delta|
&\le\left(\frac{A_0}{|\omega|}+1\right)|p(i\omega)|\\
&<\bigl(\sqrt2 A_0^{3/2}+1\bigr)|p(i\omega)|\\
&\le(2A_0^2+1)|p(i\omega)|.
\end{aligned}
\]
This proves the stated gap. In the bad case $\delta\le-1$, so the cubic is not Hurwitz; the positive gap excludes imaginary-axis roots, and therefore at least one root lies in the open right half-plane.
\end{proof}

\section{Frequency-separated affine aggregation}\label{sec:aggregation}

Let $q_1,\ldots,q_J$, where $J=m+n$, denote the $m+n$ source cubics. Their total degree is $3J$. Without the requirement of affine dependence on $x$ necessary to be able to construct our static output feedback instance in the end, simultaneous Hurwitz stability could be encoded exactly by their product. The product, however, contains higher-order products of the parameters. Our approach is to first shift conjugate copies of the cubics to widely separated imaginary frequencies, form the product but only retain the terms containing at most one parameter-dependent perturbation, and show that root counts are preserved.

Index the cubics and the two signs of their imaginary shifts by
\[
\mathcal A=\{(j,\sigma):1\le j\le J,\ \sigma\in\{-1,+1\}\}.
\]
Define the local conditioning constant
\begin{equation}\label{eq:Cstar}
C_*:=(R_0+R)^3(2A_0^2+1)
\end{equation}
and choose the integer separation
\begin{equation}\label{eq:H}
\boxed{H:=2R_0+1024\,J\,C_*M_0.}
\end{equation}
Introduce the imaginary axis shifts
\[
\Omega_j=jH,\qquad c_{j,\sigma}=i\sigma\Omega_j.
\]
For $a=(j,\sigma)\in\mathcal A$, define the shifted source cubics
\begin{equation}\label{eq:shifted-factors}
Q_a(s,x)=q_j(s-c_a,x),\quad
Q_a^0(s)=\bar q_j(s-c_a),\quad
\Delta_a=Q_a-Q_a^0.
\end{equation}
Let
\[
P_0(s)=\prod_{a\in\mathcal A}Q_a^0(s).
\]
The aggregated affine polynomial is
\begin{equation}\label{eq:P-definition}
P(s,x)
=P_0(s)+\sum_{a\in\mathcal A}
\Delta_a(s,x)\frac{P_0(s)}{Q_a^0(s)}.
\end{equation}
Equivalently, $P$ is the first-order truncation, in the perturbations $\Delta_a$, of the exact product
\[
\prod_{a\in\mathcal A}Q_a(s,x)
=\prod_{a\in\mathcal A}\bigl(Q_a^0(s)+\Delta_a(s,x)\bigr):
\]
it retains the constant term and the terms containing exactly one $\Delta_a$, and discards all terms containing two or more perturbations. 
Because the two frequency copies are conjugate, $P$ has real integer coefficients. Every $\Delta_a$ is affine in $x$ and has degree at most two in $s$, so $P$ is monic and affine in $x$ with degree
\begin{equation}\label{eq:D}
D=6J=6(m+n).
\end{equation}

For a fixed cluster $a$, put
\begin{equation}\label{eq:cluster-disks}
D_a=\{s:\abs{s-c_a}<R_0\},
\qquad D_a^+=D_a\cap\{s:\Rea s>0\},
\end{equation}
and
\begin{equation}\label{eq:Fa}
G_a(s)=\prod_{b\ne a}Q_b^0(s),
\qquad F_a(s,x)=Q_a(s,x)G_a(s).
\end{equation}
Since $P_0=Q_a^0G_a$ and $Q_a=Q_a^0+\Delta_a$, direct cancellation gives the exact relative-error identity
\begin{equation}\label{eq:relative-error}
\frac{P-F_a}{F_a}
=\frac{Q_a^0}{Q_a}
\sum_{b\ne a}\frac{\Delta_b}{Q_b^0}.
\end{equation}
The path forward now is that we want to use Rouch\'e's theorem in Lemma \ref{lem:root-transfer} to compare $P$ and $F_a$ on the boundary of $D_a$ and infer roots of $P$ from $F_a$. For this we have to bound the right-hand side of \eqref{eq:relative-error}. 
The next lemma gives a uniform bound on the relative error of the perturbation $\Delta_b$ for $b\ne a$.

\begin{lemma}[Normalized distant estimate]\label{lem:distant-infty}
Let $M=\norm{x}_\infty$. For every source index $j$ and every local variable $z$ with $\abs z\ge2R$,
\[
\abs{\frac{\Delta_j(z,x)}{\bar q_j(z)}}
\le\frac{64M}{\abs z}.
\]
\end{lemma}

\begin{proof}
For an original triple, \eqref{eq:triple-pert} and $|x_i|\le M$ give, when $|z|\ge N$,
\[
|\Delta_j(z,x)|\le8M|z|^2.
\]
For an anchor, $|x_i(z^2-Nz)|\le2M|z|^2$. Every baseline is monic cubic and all its roots lie in $|s|\le R$, so for $|z|\ge2R$,
\[
|\bar q_j(z)|\ge(|z|-R)^3\ge\frac{|z|^3}{8}.
\]
Division gives the stated common constant $64$.
\end{proof}

On every local contour, the baseline roots satisfy $|\rho|\le R$, and hence
\[
|Q_a^0(s)|\le(R_0+R)^3.
\]
On the outer circle $|s-c_a|=R_0$, Lemma~\ref{lem:root-disks} gives $|Q_a(s,x)|\ge1$. On an imaginary-axis diameter in one of the cases covered by Lemma~\ref{lem:imag-gap}, $|Q_a(s,x)|\ge\mu$. Therefore
\begin{equation}\label{eq:conditioning}
\abs{\frac{Q_a^0}{Q_a}}\le C_*
\end{equation}
on every full-disk contour for $x\in\mathcal B$, and on every relevant right-half-disk contour.

Distinct centers are separated by at least $H>2R_0$. If $s\in\overline D_a$ and $b\ne a$, then
\[
|s-c_b|\ge H-R_0>1024JC_*M_0>2R.
\]
Combining \eqref{eq:relative-error}, Lemma~\ref{lem:distant-infty}, and \eqref{eq:conditioning} gives the uniform strict estimate
\begin{equation}\label{eq:bounded-rouche}
\abs{\frac{P-F_a}{F_a}}
<C_*(2J)\frac{64M_0}{1024JC_*M_0}
=\frac18<1.
\end{equation}

\begin{proposition}[Bounded frequency compiler]\label{prop:aggregation}
With a cluster separation $H$ chosen by \eqref{eq:H}:
\begin{enumerate}[label=(\roman*)]
\item at the satisfying integer realization $x_{\pi(k)}=2k$, $P(\cdot,x)$ is Hurwitz;
\item if $x\in\mathcal B$ and the source cubic selected by Lemma~\ref{lem:global-defect} has $\Delta\le-1$, then $P(\cdot,x)$ has an open-right-half-plane root.
\end{enumerate}
\end{proposition}

\begin{proof}
Every baseline root is within $R$ of its own center, while the cluster separation gives $|s-c_b|>R$ for $s\in D_a$ and $b\ne a$; hence $G_a$ has no roots in $D_a$. Estimate~\eqref{eq:bounded-rouche} and Lemma~\ref{lem:root-transfer} first show that $P$ and $F_a=Q_aG_a$ have the same number of roots in $D_a$. Since $G_a$ has no roots there and $Q_a$ has exactly three roots in $D_a$, $F_a$ and $Q_a$ have the same count, namely three. Thus $P$ has three roots in every full cluster disk. The $2J$ disjoint disks therefore account for all $6J$ roots of $P$.

At a satisfying realization, Lemma~\ref{lem:imag-gap} allows the same comparison on every right half-disk $D_a^+$. The shifted source cubic $Q_a$ has no root there, so neither does $P$. Hence all roots of $P$ lie in the open left half-plane.

For (ii), choose the bad source from Lemma~\ref{lem:global-defect} and either frequency copy. Lemma~\ref{lem:imag-gap} gives an open-right-half-plane root of $Q_a$ and excludes roots on the imaginary-axis diameter. The same half-disk Rouch\'e comparison transfers a positive right-half-plane root count to $P$.
\end{proof}

The bounded compiler already gives a useful intermediate complexity result. Since $N=2(n+1)$, the relevant box is
\[
\mathcal C=[0,2(n+1)]^n\subset\mathcal B.
\]

\begin{corollary}[Bounded affine Hurwitz-family feasibility]\label{cor:affine-hurwitz}
The following decision problem is NP-hard: given $p_0,p_1,\ldots,p_n\in\Q[s]$ such that
\[
P(s,x)=p_0(s)+\sum_{\ell=1}^n x_\ell p_\ell(s)
\]
is monic with degree independent of $x$, decide whether
\[
\exists x\in\mathcal C:\ P(\cdot,x)\text{ is Hurwitz}.
\]
\end{corollary}

\begin{proof}
For a yes-instance, the satisfying realization $x_{\pi(k)}=2k$ lies in $\mathcal C$ and Proposition~\ref{prop:aggregation}(i) applies. For a no-instance and any $x\in\mathcal C$, no anchor is needed to detect an out-of-range variable; unsatisfiability itself implies that some original triple has $(x_b-x_a)(x_c-x_b)\le0$, hence $\Delta\le-1$. Proposition~\ref{prop:aggregation}(ii) makes the aggregate unstable. Thus the bounded family contains a Hurwitz member exactly for yes-instances. Section~\ref{sec:complexity} verifies polynomial-time exact construction.
\end{proof}

\section{Removing the auxiliary-box restriction}\label{sec:escape}

For a no-instance, the aggregate compiler already rules out stabilizing gains $x\in\mathcal B$ via Rouch\'e arguments. What remains is to rule out stabilizing gains outside the auxiliary analysis box $\mathcal B$. This is necessary because the aggregate is the first-order truncation of a product, not the exact product of the source cubics. The key is that every variable has its own anchor source and the anchor perturbation is
\[
h(s)=s(s-N)
\]
and this unstable root will dominate for sufficiently large $x_i$.

The following lemma is the final step in the chain of equivalences.
\begin{lemma}[Instability outside the auxiliary box]\label{lem:escape}
If
\[
\norm{x}_\infty\ge M_0=100N,
\]
then $P(\cdot,x)$ has an open-right-half-plane root.
\end{lemma}

\begin{proof}
Let $M=\norm{x}_\infty\ge M_0$ and choose $i$ with $|x_i|=M$. Select either shifted copy $a$ of the anchor for $i$, write $w=s-c_a$, and consider
\[
\mathcal E_a=\left\{s:\abs{w-N}<\frac N4\right\}.
\]
This disk lies in the open right half-plane. Since $h(w)=w(w-N)$, it contains exactly one root of $h$. On the boundary,
\[
\frac{3N}{4}\le|w|\le\frac{5N}{4},
\qquad
|h(w)|\ge\frac{3N^2}{16}.
\]
For $N\ge4$,
\[
|\bar q_A(w)|
\le |w|^3+N|w|^2+2N^2|w|+2N^3+1
\le9N^3,
\]
and hence
\begin{equation}\label{eq:escape-local-ratio}
\abs{\frac{\bar q_A(w)}{h(w)}}\le48N.
\end{equation}

First consider the selected shifted anchor factor
\[
Q_a(s,x)=\bar q_A(w)+x_i h(w).
\]
Set
\[
\eta:=\abs{\frac{\bar q_A(w)}{x_i h(w)}}
\le\frac{48N}{M}
\le\frac{12}{25}<1
\]
on $\partial\mathcal E_a$. Rouch\'e's theorem therefore shows that $Q_a$ has the same number of roots in $\mathcal E_a$ as $x_i h(w)$, namely one. Moreover,
\[
\abs{\frac{Q_a^0}{Q_a}}
\le\frac{\eta}{1-\eta}
<2\eta\qquad(\eta\le12/25<1/2)
\le\frac{96N}{M}
\]
on $\partial\mathcal E_a$.

Since $\overline{\mathcal E_a}\subset D_a$ because $5N/4<R_0$, the frequency-separation argument above shows that $G_a$ is nonzero on $\overline{\mathcal E_a}$. Thus $F_a=Q_aG_a$ also has exactly one root in $\mathcal E_a$. For every $b\ne a$,
\[
|s-c_b|\ge H-\frac{5N}{4}>H-R_0,
\]
and Lemma~\ref{lem:distant-infty} gives
\[
\sum_{b\ne a}
\abs{\frac{\Delta_b}{Q_b^0}}
\le\frac{128JM}{H-R_0}.
\]
Using the relative-error identity \eqref{eq:relative-error},
\[
\begin{aligned}
\abs{\frac{P-F_a}{F_a}}
&\le\frac{96N}{M}\frac{128JM}{H-R_0}\\
&<\frac{12288JN}{1024JC_*M_0}
 =\frac{3}{25C_*}
 \le\frac3{25}<1.
\end{aligned}
\]
 A second application of Rouch\'e's theorem transfers the unique root of $F_a$ in $\mathcal E_a$ to $P$; since $\mathcal E_a$ lies in the open right half-plane, $P$ has an open-right-half-plane root.
\end{proof}

The crucial equivalence is established.
\begin{proposition}[Unrestricted affine-family equivalence]\label{prop:one-poly}
The aggregate satisfies
\[
\begin{aligned}
&\text{the Betweenness instance is satisfiable}\\
&\qquad\Longleftrightarrow\quad
\exists x\in\R^n:\ P(\cdot,x)\text{ is Hurwitz}.
\end{aligned}
\]
\end{proposition}

\begin{proof}
For a yes-instance, use the satisfying realization $x_{\pi(k)}=2k$ and Proposition~\ref{prop:aggregation}(i). Conversely, suppose the instance is unsatisfiable and let $x\in\R^n$. If $\norm{x}_\infty\ge M_0$, Lemma~\ref{lem:escape} gives an open-right-half-plane root. If $\norm{x}_\infty<M_0$, Lemma~\ref{lem:global-defect} supplies a source cubic with $\Delta\le-1$ and constant coefficient at least one, and Proposition~\ref{prop:aggregation}(ii) again gives an open-right-half-plane root. Thus no real parameter vector makes the aggregate Hurwitz.
\end{proof}

\section{Realization as static output feedback}\label{sec:realization}

The affine aggregate can be expanded as
\[
P(s,x)=s^D+\sum_{k=0}^{D-1}
\left(p_{k0}+\sum_{\ell=1}^n p_{k\ell}x_\ell\right)s^k,
\]
and hence admits a standard companion realization. It is much more instructive, however, to construct the plant directly from the source cubics.

For a strictly proper fraction $r(s)/d(s)$ with
\[
d(s)=s^3+d_2s^2+d_1s+d_0,
\qquad
r(s)=r_0+r_1s+r_2s^2,
\]
recall its companion realization
\begin{equation}\label{eq:comp-transfer}
\frac{r(s)}{d(s)}=c_{\mathrm{comp}}(r)(sI-A_{\mathrm{comp}}(d))^{-1}b_{\mathrm{comp}}.
\end{equation}
with state-space data
\[
A_{\mathrm{comp}}(d)=
\begin{bmatrix}
0&1&0\\
0&0&1\\
-d_0&-d_1&-d_2
\end{bmatrix},\;
b_{\mathrm{comp}}=\begin{bmatrix}0\\0\\1\end{bmatrix},
\;
c_{\mathrm{comp}}(r)^\top=\begin{bmatrix}r_0\\r_1\\r_2\end{bmatrix}.
\]

For source $j$ let $A_j^0=A_{\mathrm{comp}}(\bar q_j)$ and define the real frequency-shifted block
\begin{equation}\label{eq:realify}
\mathcal A_j=
\begin{bmatrix}
A_j^0&-\Omega_jI_3\\
\Omega_jI_3&A_j^0
\end{bmatrix},
\qquad
\mathcal B_j=\begin{bmatrix}b_{\mathrm{comp}}\\0\end{bmatrix}.
\end{equation}
For every parameter $x_i$ occurring in source $j$, let $h_{ji}(s)$ denote its coefficient polynomial in $q_j-\bar q_j$, and set
\[
\mathcal C_{ji}=2[\,c_{\mathrm{comp}}(h_{ji})\quad 0_{1\times3}\,].
\]
For an original triple $t=(a,b,c)$,
\begin{equation}\label{eq:hrows}
 h_{t,a}=-(s^2+N^2),\qquad
 h_{t,b}=s^2-Ns,\qquad
 h_{t,c}=Ns+N^2,
\end{equation}
while for the anchor of $i$,
\[
h_{i,i}^{A}=s^2-Ns.
\]

Order the $J=m+n$ source blocks arbitrarily and set
\begin{equation}\label{eq:direct-ABC}
A=\operatorname{diag}(\mathcal A_1,\ldots,\mathcal A_J),
\qquad
B=\operatorname{col}(\mathcal B_1,\ldots,\mathcal B_J).
\end{equation}
Define $C\in\mathbb Z^{n\times6J}$ blockwise by
\begin{equation}\label{eq:C-block}
C_{i,\text{ block }j}=-\mathcal C_{ji},
\end{equation}
with a zero block when $x_i$ does not occur in source $j$.

Applying the block-diagonalization identity derived in \ref{app:direct-realization} gives
\[
\mathcal C_{ji}(sI-\mathcal A_j)^{-1}\mathcal B_j
=
\frac{h_{ji}(s-i\Omega_j)}{\bar q_j(s-i\Omega_j)}
+
\frac{h_{ji}(s+i\Omega_j)}{\bar q_j(s+i\Omega_j)}.
\]
Consequently, if $g_i(s)$ denotes the sum of these two shifted transfer terms over all source blocks containing $x_i$, then
\[
C(sI-A)^{-1}B=-[g_1(s),\ldots,g_n(s)]^\top.
\]
Moreover $\det(sI-A)=P_0(s)$. Hence, for
\[
K=[x_1\ \cdots\ x_n],
\]
the matrix determinant lemma gives
\begin{equation}\label{eq:char-poly}
\boxed{
\det\bigl(sI-(A+BKC)\bigr)
=P_0(s)\left(1+\sum_{i=1}^n x_i g_i(s)\right)
=P(s,x).
}
\end{equation}
The determinant-lemma calculation is initially valid for $s\notin\sigma(A)$. Both sides are polynomials in $s$, so the identity then holds identically.
Thus the direct realization has state dimension $D=6(m+n)$. All entries of $A,B,C$ are integers: the baseline companion coefficients are integer polynomials in $N$ of degree at most three, the frequency shifts are integer, and the numerator rows in \eqref{eq:hrows} have integer coefficients. For fixed $n,m$, the matrices $A$ and $B$ depend only on the prescribed block types and frequencies; the actual Betweenness incidence pattern enters through the sparse block pattern of $C$.

\section{Exact construction and main result}\label{sec:complexity}

Recall that $L$ is the binary encoding length of the Betweenness instance. Since $n,m,J=O(L)$ and $N=O(L)$, the explicit constants satisfy
\[
\begin{aligned}
M_0&=O(N),& A_0&=O(N),\\
R_0&=O(N),
\end{aligned}
\]
\[
\begin{aligned}
\mu^{-1}&=O(N^2),& C_*&=O(N^5),\\
H&=O(JN^6),& \Omega_J&=O(J^2N^6).
\end{aligned}
\]
Therefore every integer appearing directly in $A,B,C$ has $O(\log L)$ bits.
The matrices are obtained by evaluating these explicit integer formulas and placing the resulting blocks according to the incidence pattern of the input triples. Since their dimensions are polynomial in $L$ and all arithmetic uses polynomially many operations on $O(\log L)$-bit integers, this construction runs in polynomial time.

\begin{proposition}[Polynomial-time exact construction]\label{prop:coeff-growth}
The direct matrices $(A,B,C)$ can be constructed exactly in polynomial time. Every matrix entry has $O(\log L)$ bits, and their complete dense encoding has $O(L^2\log L)$ bits. If the coefficient representation of $P$ is required for the separate bounded affine-family problem, it can also be computed exactly in polynomial time and has polynomial encoding length.
\end{proposition}

\begin{proof}
Every entry of $A$ is zero, an integer coefficient of one of the two baseline companions, or $\pm\Omega_j$; every entry of $B$ is zero or one; and every entry of $C$ is zero or one of the integer coefficients $\{-N^2,-N,-1,0,1,N,N^2\}$ multiplied by the fixed factor two from the output row. Hence every direct matrix entry has $O(\log L)$ bits. The dimensions are $D\times D$, $D\times1$, and $n\times D$ with $D=6J=O(L)$, so the dense output uses $O(L^2)$ entries and $O(L^2\log L)$ bits.

For an explicit coefficient list of $P$, equation~\eqref{eq:P-definition} contains $O(L)$ products of $O(L)$ cubic factors. Their shifted coefficients have polynomial magnitude and $O(\log L)$ bit length. Exact polynomial multiplication therefore computes all coefficients using polynomially many integer operations with polynomially bounded intermediate bit lengths. This expansion is not needed for the SOF construction itself though.
\end{proof}

The main theorem of the paper now follows immediately from Propositions~\ref{prop:one-poly} and \ref{prop:coeff-growth}.
\begin{proof}[Proof of Theorem~\ref{thm:main}]
Proposition~\ref{prop:one-poly} and \eqref{eq:char-poly} give
\[
\begin{aligned}
\text{Betweenness satisfiable}
&\Longleftrightarrow
\exists x\in\R^n:P(\cdot,x)\text{ Hurwitz}\\
&\Longleftrightarrow
\exists K\in\R^{1\times n}:A+BKC\text{ Hurwitz}.
\end{aligned}
\]
Proposition~\ref{prop:coeff-growth} proves polynomial-time exact construction and the stated encoding bounds. Since Betweenness is NP-complete, this is a polynomial-time many-one reduction to $\mathsf{SOF}_1$.
\end{proof}

\section{Conclusion}

We have given a polynomial-time reduction from Betweenness to unrestricted single-input, multiple-output static output-feedback stabilization. The reduction first represents each object by an unrestricted real parameter and uses a common cubic source gadget to encode both the ordering constraints and auxiliary anchor constraints. The anchors serve two purposes: they localize every feasible source realization, and they supply the mechanism that excludes spurious stabilizing gains of large magnitude after aggregation.

The central technical step combines all source conditions into one real monic polynomial whose coefficients remain affine in the gain parameters. Frequency separation permits local root-count comparisons between the aggregate and its source factors on a bounded region, while the anchor-based escape argument handles the complementary unbounded region. A bank of realified shifted cubic systems realizes the resulting affine polynomial directly as a SIMO static output-feedback characteristic polynomial.

The construction has integer data, state dimension $6(m+n)$, and logarithmic entry bit lengths. It therefore establishes NP-hardness for unrestricted SIMO stabilization; transposition gives the corresponding MISO result, and zero padding extends the result to general MIMO static output feedback. As a by-product, the bounded compiler proves NP-hardness of deciding whether an affinely parameterized rational polynomial family contains a Hurwitz member on a rational box.

\appendix

\section{Identity for the real frequency-shifted block}\label{app:direct-realization}

Let $A_*$ be a real square matrix, $b_*$ a real column, $c_*$ a real row, and define
\[
\mathcal A_\Omega(A_*)=
\begin{bmatrix}A_*&-\Omega I\\ \Omega I&A_*\end{bmatrix},
\quad
\mathcal B(b_*)=\begin{bmatrix}b_*\\0\end{bmatrix},
\quad
\mathcal C(c_*)=2[c_*\ \ 0].
\]
Let $I$ have the same dimension as $A_*$ and introduce the complex change of basis
\[
S=
\begin{bmatrix}
I&I\\
-iI&iI
\end{bmatrix},
\qquad
S^{-1}=\frac12
\begin{bmatrix}
I&iI\\
I&-iI
\end{bmatrix}.
\]
Direct block multiplication gives
\[
S^{-1}\mathcal A_\Omega(A_*)S
=
\begin{bmatrix}
A_*+i\Omega I&0\\
0&A_*-i\Omega I
\end{bmatrix},
\]
and transforms the input and output as
\[
S^{-1}\mathcal B(b_*)=
\frac12\begin{bmatrix}b_*\\b_*\end{bmatrix},
\qquad
\mathcal C(c_*)S=2[c_*\ \ c_*].
\]
Thus, if $G(s)=c_*(sI-A_*)^{-1}b_*$, then
\[
\begin{aligned}
\mathcal C(c_*)\bigl(sI-\mathcal A_\Omega(A_*)\bigr)^{-1}\mathcal B(b_*)
&=c_*\bigl((s-i\Omega)I-A_*\bigr)^{-1}b_*\\
&\qquad{}+c_*\bigl((s+i\Omega)I-A_*\bigr)^{-1}b_*\\
&=G(s-i\Omega)+G(s+i\Omega).
\end{aligned}
\]
The same similarity transformation gives
\[
\det\bigl(sI-\mathcal A_\Omega(A_*)\bigr)
=\det((s-i\Omega)I-A_*)\det((s+i\Omega)I-A_*).
\]
Applying these identities with $A_*=A_{\mathrm{comp}}(\bar q_j)$, $b_*=b_{\mathrm{comp}}$, and $c_*=c_{\mathrm{comp}}(h_{ji})$ proves the transfer and determinant statements used in Section~\ref{sec:realization}. The factors in the determinant identity are the two conjugate frequency-shifted baseline cubics, while the transfer identity gives the corresponding pair of shifted numerator terms.

\section*{Acknowledgments}
The author acknowledges the financial support from the Swedish Government Agency VINNOVA through the SEDDIT competence center program.

\section*{Declaration of generative AI and AI-assisted technologies
in the manuscript preparation process}
OpenAI ChatGPT 5.6 and Claude Opus 5 were used in various stages of the manuscript preparation, including ideation, generation of numerical verification code, writing assistance, and adversarial review. The author assumes responsibility for all content.

\bibliographystyle{siamplain}
\bibliography{references}

\end{document}